\documentclass{birkjour}
\usepackage{amsmath}
\usepackage{amssymb}
\usepackage{amsthm}
\usepackage{eqlist}
\usepackage{enumerate}
\usepackage[mathscr]{eucal}
\usepackage{hyperref}
\hypersetup{
     colorlinks   = true,
     citecolor    = blue
}
\hypersetup{linkcolor=blue}

 \newtheorem{theorem}{Theorem}[section]
 \newtheorem{corollary}[theorem]{Corollary}
 \newtheorem{lemma}[theorem]{Lemma}
 \newtheorem{proposition}[theorem]{Proposition}
 \theoremstyle{definition}
 \newtheorem{definition}[theorem]{Definition}
 \theoremstyle{remark}
 
 \newtheorem{example}[theorem]{Example}
 
 \numberwithin{equation}{section}

\begin{document}
\title[Geometry of $\mathcal{P}\mathcal{R}$-semi-invariant warped product submanifolds]
{Geometry of $\mathcal{P}\mathcal{R}$-semi-invariant warped product submanifolds in paracosymplectic manifold}
\author[S. K. Srivastava] {S. K. Srivastava}
\address{Department of Mathematics,\br
                   Central University of Himachal Pradesh,\br
                   Dharamshala-176215,\br
                   Himachal Pradesh,\br
	       INDIA.}
\email{sachin@cuhimachal.ac.in}
\author{A. Sharma}
\address{Department of Mathematics,\br
                   Central University of Himachal Pradesh,\br
                   Dharamshala-176215,\br
                   Himachal Pradesh,\br
	       INDIA.}
\email{cuhp13rdmath01@cuhimachal.ac.in}
\thanks {S. K. Srivastava: partially supported through the UGC-BSR Start-Up-Grant vide
their letter no. F.30-29/2014(BSR). A. Sharma: supported by Central University of Himachal Pradesh through the Research fellowship for Ph.D.} 

\begin{abstract}
The purpose of this paper is to study $\mathcal{P}\mathcal{R}$-semi-invariant warped product submanifolds of a paracosymplectic manifold $\widetilde{M}$. We prove that the distributions associated with the definition of $\mathcal{P}\mathcal{R}$-semi-invariant warped product submanifold $M$ are always integrable. A necessary and sufficient condition for an isometrically immersed $\mathcal{P}\mathcal{R}$-semi-invariant submanifold of $\widetilde{M}$ to be a $\mathcal{P}\mathcal{R}$-semi-invariant warped product submanifold is obtained in terms of the shape operator.  
\end{abstract}
\maketitle
\section{Introduction}\label{intro}
The concept of warped product (or, more generally warped bundle) is one of the most effective generalizations of pseudo-Riemannian  products \cite{OB}. The premise of warped product has perceived several important contributions in complex and contact Riemannian (or pseudo-Riemannian) geometries, and has been successfully applied in Hamiltonian spaces, general relativity and black holes (c.f., \cite{HM,JC, JCH, ST}). 

The study of warped product was initiated by Bishop-Neill \cite{RB}. However, the consideration has attained momentum when Chen introduced the notion of CR­-warped product in Kaehlerian manifold $\widetilde{N}$ and proved the non-existence of proper warped product CR-submanifolds in the form $N_{T} \times {}_fN_{\bot}$ such that $N_{T}$ is a holomorphic submanifold and $N_{\bot}$ is a totally real submanifold of $\widetilde{N}$ \cite{BY}. Subsequently, Hasegawa-Mihai \cite{IH} and Munteanu \cite{MI} continued the study for Sasakian manifold that can be viewed as an odd-dimensional analogue of $\widetilde{N}$. Further, several geometers have studied the existence and non-existence of warped product submanifolds in almost contact and Lorentzian manifolds (c.f., \cite{BYC, KA, KHAN, SY, UDDIN}). Recently in \cite{CPR}, Chen-Munteanu brought our attention to the geometry of $\mathcal{P}\mathcal{R}$-warped products in para-K\"{a}hler manifolds and obtained some basic results on such submanifolds. 

This paper is organized as follows. In Sect. \ref{pre}, the basic information about almost paracontact metric manifolds, paracosymplectic manifolds and submanifolds is given. In Sect. \ref{wps}, we proved the non-existence of a proper warped product submanifold of a paracosymplectic manifold $\widetilde{M}$ in the form $B\times_{f} F$ such that the characteristic vector field $\xi$ is tangent to $F$, where $f$ is a warping function. In Sect. \ref{pr}, we study $\mathcal{P}\mathcal{R}$-semi-invariant warped product submanifolds of $\widetilde{M}$ and found the distributions concerned with the definition of $\mathcal{P}\mathcal{R}$-semi-invariant submanifold $M$ are integrable. Further, we obtained a necessary and sufficient condition for an isometrically immersed submanifold $M$ of $\widetilde{M}$ to be a $\mathcal{P}\mathcal{R}$-semi-invariant warped product submanifold. Finally, we gave an example of a $\mathcal{P}\mathcal{R}$-semi-invariant submanifold $F\times_{f} B$ of a paracosymplectic manifold in Sect. \ref{ex}.
\section{Preliminaries}$\label{pre}$
\subsection{Almost paracontact metric manifolds}
A $(2n + 1)$-dimensional $C^{\infty }$ manifold $\widetilde{M}$ has an  almost paracontact structure $(\phi ,\xi ,\eta )$, if it admits a tensor field $\phi$ of type (1, 1), a vector field $\xi$, and a $1$-form $\eta$ on  $\widetilde{M}$, satisfying conditions:
\begin{align}
  &\phi ^{2} =Id-\eta \otimes \xi ,\quad\eta \left(\xi \right)=1\label{GrindEQ__2_1_}
\end{align}
where $Id$ is the identity transformation and the tensor field $\phi$ induces an almost paracomplex structure on the distribution $D$ = ker($\eta$), that is the eigen distributions $D^{\pm}$ corresponding to the eigenvalues $\pm1$, have equal dimensions dim$D^{+}$= dim$D^{-}= n$.
From the equation (\ref{GrindEQ__2_1_}), it can be easily deduced that 
\begin{align}\label{phixi}
\phi\xi = 0, \quad \eta\circ\phi = 0 \quad {\rm and \quad rank}(\phi)=2n.
\end{align}
The manifold $\widetilde{M}$ is said to be an almost paracontact manifold if it is endowed with an almost paracontact structure (c.f., \cite{SK,ZA}).
If an almost paracontact manifold $\widetilde{M}$ admits a pseudo-Riemannian metric $g$ satisfying:
\begin{align}
g\left(X,Y\right)=-g\left(\phi X,\phi Y\right)+\eta \left(X\right)\,\eta \left(Y\right),\label{GrindEQ__2_2_}
\end{align}
where signature of $g$ is necessarily $(n+1,\,n)$ for any vector fields $X$ and $Y$; then the quadruple $(\phi,\xi,\eta,g)$ is called an almost paracontact metric structure and the manifold $\widetilde{M}$ equipped with paracontact metric structure is called an almost paracontact metric manifold. With respect to $g$, $\eta$ is metrically dual to $\xi$, that is
\begin{align}\label{gx}
g(X,\xi)=\eta(X). 
\end{align}
With the consequences of Eqs. \eqref{GrindEQ__2_1_} and \eqref{phixi}, Eq. (\ref{GrindEQ__2_2_}) implies that
\begin{align}\label{GrindEQ__2_3_}
g(\phi X,Y)=-g(X,\phi Y),
\end{align}
for any $X,Y\in \Gamma(T\widetilde{M})$; $\Gamma(T\widetilde{M})$ denote the sections of tangent bundle $T\widetilde M$ of $\widetilde M$, that is, the space of vector fields on $\widetilde M$. The fundamental $2$-form $\Phi$ on $\widetilde{M}$ is given by 
\begin{align}\label{PHI} g\left(X,\phi Y\right)=\Phi\left(X,Y\right). \end{align}
\noindent An almost paracontact metric structure-$(\phi,\xi,\eta,g)$ is para-Sasakian if and
only if 	
\begin{align}\label{paras}
(\widetilde\nabla_{X}\phi) Y=-g(X,Y)\xi+\eta(Y)X.
\end{align}
From \eqref{phixi}, \eqref{GrindEQ__2_2_} and \eqref{paras}, it can be easily deduced for a para-Sasakian manifold that
\begin{equation}\label{nablax}
\widetilde\nabla_{X}\xi = -\phi X,\quad \widetilde\nabla_{\xi}\xi=0.
\end{equation}
In particular, a para-Sasakian manifold is $K$-paracontact \cite{ZA}.
\begin{definition}
An almost paracontact metric manifold $\widetilde{M}(\phi, \xi, \eta, g)$ is said to be 
\begin{itemize}
\item[$\bullet$] paracosymplectic if the forms $\eta$ and $\Phi$ are parallel with respect to the Levi-Civita connection $\widetilde{\nabla }$ on $\widetilde{M}$, i.e., 
\begin{align}\widetilde{\nabla }\eta=0\quad{\rm and}\quad \widetilde{\nabla }\Phi=0.\label{GrindEQ__2_4_}\end{align}
\item[$\bullet$] an almost paracosymplectic if the forms $\eta$ and $\Phi$ are closed, i.e., $d\eta=0$ and $d\Phi=0$ (see \cite{PD, IK}).
\end{itemize}
\end{definition}
\noindent Now, we give an example of a paracosymplectic manifold:
\begin{example}\label{ex1}

\noindent We consider the $5$-dimensional manifold $\widetilde M=\mathbb{R}^4\times\mathbb{R}_{+}\subset\mathbb{R}^5$ with the standard Cartesian coordinates $(x_1,x_2,y_1,y_2,t)$. Define the structure $(\phi,\xi,\eta)$ on $\widetilde M$ by 
\begin{align}\label{strucpc}
\phi e_{1}=e_{3},\ \phi e_{2}=e_{4},\ \phi e_{3}=e_{1},\  
 \phi e_{4}=e_{2},\ \phi e_{5}=0,\ \xi=e_{5},\ \eta =dt,
\end{align}
where $e_{1} = \dfrac{\partial}{\partial x_1},\,e_{2} = \dfrac{\partial}{\partial x_2},\,e_{3} = \dfrac{\partial}{\partial y_1},\,e_{4} = \dfrac{\partial}{\partial y_2}$ and $e_{5} = \dfrac{\partial}{\partial t}$.
Consider $g$ to be the pseudo-Riemannian metric defined by
\begin{align}\label{metricex}
[g(e_{i},e_{j})]=\begin{bmatrix}
                         x^{2}&0&0&0&0\\
                         0&y^{2}&0&0&0\\
                         0&0&-x^{2}&0&0\\
                         0&0&0&-y^{2}&0\\
                         0&0&0&0&1
\end{bmatrix}.
\end{align}
Then by straightforward calculations, one verifies that the structure $(\phi,\xi,\eta,g)$ is an almost paracontact metric structure. For the Levi-Civita connection $\widetilde\nabla$ with respect to pseudo-Riemannian metric $g$, we obtain
\begin{align*}
\begin{matrix}
\widetilde\nabla_{e_{1}}{e_{1}}=\frac{1}{x}e_{1},  &\widetilde\nabla_{e_{1}}{e_{2}}=0, &\widetilde\nabla_{e_{1}}{e_{3}}=\frac{1}{x}e_{3}, &\widetilde\nabla_{e_{1}}{e_{4}}=0,  &\widetilde\nabla_{e_{1}}{e_{5}}=0, \\ 
\widetilde\nabla_{e_{2}}{e_{1}}=0, &\widetilde\nabla_{e_{2}}{e_{2}}=\frac{1}{y}e_{2}, &\widetilde\nabla_{e_{2}}{e_{3}}=0, &\widetilde\nabla_{e_{2}}{e_{4}}=\frac{1}{y}e_{4}, &\widetilde\nabla_{e_{2}}{e_{5}}=0,  \\
\widetilde\nabla_{e_{3}}{e_{1}}=\frac{1}{x}e_{3}, &\widetilde\nabla_{e_{3}}{e_{2}}=0, &\widetilde\nabla_{e_{3}}{e_{3}}=\frac{1}{x}e_{1}, &\widetilde\nabla_{e_{3}}{e_{4}}=0, &\widetilde\nabla_{e_{3}}{e_{5}}=0, \\
\widetilde\nabla_{e_{4}}{e_{1}}=0, &\widetilde\nabla_{e_{4}}{e_{2}}=\frac{1}{y}e_{4}, &\widetilde\nabla_{e_{4}}{e_{3}}=0, &\widetilde\nabla_{e_{4}}{e_{4}}=\frac{1}{y}e_{2}, &\widetilde\nabla_{e_{4}}{e_{5}}=0, \\
\widetilde\nabla_{e_{5}}{e_{1}}=0, &\widetilde\nabla_{e_{5}}{e_{2}}=0, &\widetilde\nabla_{e_{5}}{e_{3}}=0, &\widetilde\nabla_{e_{5}}{e_{4}}=0, &\widetilde\nabla_{e_{5}}{e_{5}}=0. 
\end{matrix}
\end{align*}
From the above computations, it can be easily seen that $\widetilde{M}(\phi,\xi,\eta,g)$ is a paracosymplectic manifold.\end{example}

\subsection{Geometry of submanifolds}
Let $M$ be a submanifold immersed in a $(2n + 1)$-dimensional almost paracontact manifold $\widetilde{M}$; we denote by the same symbol $g$ the induced metric on $M$. Let $\Gamma (TM^{\bot })$ denote the set of vector fields
normal to $M$ and $\Gamma (TM)$ the sections of tangent bundle $TM$ of $M$ then Gauss and Weingarten formulas are given by, respectively,
\begin{align} \widetilde{\nabla }_{X} Y&=\nabla _{X} Y+h(X,Y), \label{GrindEQ__2_5_}\\
\widetilde{\nabla }_{X} \zeta &=-A_{\zeta} X+\nabla _{X}^{\bot } \zeta \label{GrindEQ__2_6_}
\end{align}
for any $X,Y\in \Gamma (TM)$ and $\zeta \in \Gamma (TM^{\bot })$, where $\nabla$ is the induced connection, $\nabla ^{\bot }$ is the normal connection on the normal bundle $TM^{\bot }$, $h$ is the second fundamental form, and the shape operator $A_{\zeta}$ associated with the normal section $\zeta$ is given by
\begin{align} \label{GrindEQ__2_7_} g\left(A_{\zeta} X,Y\right)=g\left(h(X,Y),\zeta\right).\end{align}
The mean curvature vector $H$ of $M$ is given by $H = \frac{1}{n}{\rm trace}\,(h)$. A pseudo-Riemannian submanifold $M$ is said to be \cite{CPR} 
\begin{itemize}
\item[$\bullet$] {\it totally geodesic} if its second fundamental form vanishes identically. 
\item[$\bullet$] {\it umbilical} in the direction of a normal vector field $\zeta$ on $M$, if $A_{\zeta} = \lambda Id$, for certain function $\lambda$ on $M$; here $\zeta$ is called a umbilical section. 
\item[$\bullet$] {\it totally umbilical} if $M$ is umbilical with respect to every (local) normal vector field. 
\item[$\bullet$] {\it minimal} if the mean curvature vector $H$ vanishes identically. 
\item[$\bullet$] {\it quasi-minimal} if $H$ is a light-like vector field.
\end{itemize}
Consider that $M$ is an isometrically immersed submanifold of an almost paracontact metric manifold $\widetilde{M}$. For any  
$X\in \Gamma (TM)$ and $N\in \Gamma(TM^{\bot })$, if we write 
\begin{align} \phi X&=tX+nX,\label{GrindEQ__2_8_}\\
\phi N&=t'N+n'N, \label{GrindEQ__2_9_}
\end{align}
where $tX$ (resp., $nX$) is tangential (resp., normal) part of $\phi X$ and $t'N$ (resp., $n'N$) is tangential (resp., normal) part of $\phi N$. Then the submanifold $M$ is said to be {\it invariant} if $n$ is identically zero and {\it anti-invariant} if $t$ is identically zero. From Eqs. $\eqref{GrindEQ__2_3_}$ and $\eqref{GrindEQ__2_8_}$, we obtain that 
\begin{align} \label{GrindEQ__2_10_}
g(X,tY)=-g(tX,Y). 
\end{align}
\noindent Let $M$ be an immersed submanifold of a paracosymplectic manifold $\widetilde{M}$ then for any $X,Y\in\Gamma(TM)$ we obtain by use of Eqs. \eqref{GrindEQ__2_4_}, \eqref{GrindEQ__2_5_}, \eqref{GrindEQ__2_6_} and \eqref{GrindEQ__2_7_} that
\begin{align}
 \left(\nabla _{X} t\right)Y&= A_{nY}X + t'h(X,Y),\label{GrindEQ__2_16_}\\
\left(\nabla _{X} n\right)Y &= n'h(X,Y) - h(X,tY),\label{GrindEQ__2_17_}
\end{align}
where the covariant derivatives of the tensor fields $t$ and $n$ are, respectively, defined by
\begin{align}
(\nabla _{X} t)Y &=\nabla _{X} tY-t\nabla_{X}Y,\label{GrindEQ__2_11_} \\
(\nabla _{X} n)Y &=\nabla _{X}^{\bot}nY-n\nabla _{X}Y. \label{GrindEQ__2_12_}
\end{align}
\noindent The canonical structure $t$ and $n$ on a submanifold $M$ are said to be \textit{parallel} if  $\nabla t=0$ and ${\nabla }n=0$, respectively. 
From Eqs. \eqref{GrindEQ__2_4_} and \eqref{GrindEQ__2_5_}, we can easily prove the following lemma for later use:
\begin{lemma}\label{lem__3_2_}
Let $M$ be an immersed submanifold of a paracosymplectic manifold $\widetilde{M}(\phi,\xi,\eta,g)$ such that $\xi$ is tangent to $M$. Then for any $X \in \Gamma(TM)$, we have
 \begin{align}
  &\nabla_{X}\xi =0,\label{GrindEQ__3_2_}\\
&h(X,\xi)=0.\label{GrindEQ__3_3_}
 \end{align}
\end{lemma}
\section{Warped product submanifolds}\label{wps}
\noindent Let $\left(B,g_{B} \right)$ and $\left(F ,g_{F} \right)$ be two pseudo-Riemannian manifolds and ${f}$ be a positive smooth function on $B$. Consider the product manifold $B\times F$ with canonical projections 
\begin{align}\label{cp}
\pi:B \times F\to B\quad{\rm and}\quad \sigma:B \times F\to F.
\end{align}
Then the manifold $M=B \times_{f} F $ is said to be \textit{warped product} if it is equipped with the following warped metric
\begin{align}\label{wmetric}
g(X,Y)=g_{B}\left(\pi_{\ast}(X),\pi_{\ast}(Y)\right) +(f\circ\pi)^{2}g_{F}\left(\sigma_{\ast}(X),\sigma_{\ast}(Y)\right)
\end{align}
for all $X,Y\in \Gamma(TM)$ and `$\ast$' stands for derivation map, or equivalently,
\begin{align}
g=g_{B} +f^{2} g_{F}.
\end{align}
The function $f$ is called {\it the warping function} and a warped product manifold $M$ is said to be {\it trivial} if $f$ is constant \cite{RB}.

Now, we recall the following proposition for the warped product manifolds \cite{RB}:

\begin{proposition}$\label{prop__3_1_}$
For $X, Y \in \Gamma(TB)$ and $U, V \in \Gamma(TF)$, we obtain on warped product manifold $M=B\times_{f} F$ that
\begin{itemize}
\item[(1)] 	$\nabla _{X}Y \in \Gamma(TB),$
\item[(2)]	$\nabla _{X}U =\nabla _{U}X=X(\ln{f})U,$
\item[(3)]	$\nabla _{U}V =\nabla^{'}_{U}V-g(U,V)grad(\ln{f}),$
\end{itemize}
where $\nabla$ and $\nabla^{'}$ denotes the Levi-Civita connections on $M$ and $F$ respectively.
\end{proposition}
\noindent For a warped product $M=B \times_{f} F$; $B$ is totally geodesic and $F$ is totally umbilical in $M$ \cite{RB}.

In \cite{PE}, Ehrlich introduced a notion of doubly warped product to generalize the warped product. Let us consider the product manifold $B\times F$ with canonical projections given by \eqref{cp}. Then a doubly warped product of pseudo-Riemannian manifolds of $\left(B,g_{B} \right)$ and $\left(F ,g_{F} \right)$ with smooth warping functions $f_{1}:B\to (0,\infty)$ and $f_{2}:F\to (0,\infty)$ is a manifold ${}_{f_{2} } B \times_{f_{1} } F$ endowed with the following doubly warped metric

\begin{align}\label{dwmetric}
g(X,Y)=(f_{2}\circ\sigma)^{2}g_{B}\left(\pi_{\ast}(X),\pi_{\ast}(Y)\right) +(f_{1}\circ\pi)^{2}g_{F}\left(\sigma_{\ast}(X),\sigma_{\ast}(Y)\right)
\end{align}
for all $X,Y\in \Gamma(TM)$, or equivalently,
\begin{align}
g=f_{2} ^{2} g_{B} +f_{1} ^{2} g_{F}.
\end{align}
If either $f_1 = 1$ or $f_2 = 1$, but not both, then ${}_{f_{2} } B \times_{f_{1} } F$ becomes a warped product. If $f_1 =f_2= 1$, then we have a product manifold. If neither $f_1 $ nor $f_2 $ is constant, then we obtain a proper (non trivial) doubly warped product manifold (see also \cite{SY, BU}).
In this case formula $(2)$ of proposition \ref{prop__3_1_} is generalized as 
\begin{align} \label{GrindEQ__3_1_}\nabla _{X} V=\left(X\ln f_{1} \right)V+\left(V\ln f_{2} \right)X
\end{align}
for each $X\in \Gamma(TB) $ and $V \in \Gamma(TF)$ \cite{MI}. For the proper doubly warped product manifold $M={}_{f_{2} } B \times_{f_{1} } F$, we have from \cite{BU} that the: 
\begin{itemize}
\item[(i)] leaves $B\times \{q\}$ and the fibers $\{p\}\times F$ of $M$ are totally umbilical and
\item[(ii)] leaf $B\times \{q\}$ (resp., fiber $\{p\}\times F$) is totally geodesic if ${\rm grad}_{F}(f_{2})\vert_{q}=0$ (resp., ${\rm grad}_{B}(f_{1})\vert_{p}=0$). 
\end{itemize} 
Presently we will prove the following theorem:
\begin{theorem}$\label{theorem__3_5_}$
There do not exist a proper warped product submanifold $M=B \times_{f} F$ of a paracosymplectic manifold $\widetilde{M}(\phi,\xi,\eta,g)$ such that $\xi$ have both $TB$ and $TF$ components.
\end{theorem}
\begin{proof}
For $\xi \in \Gamma(TM)$ we can write $\xi=\xi_1+\xi_2$, where $\xi_1 \in \Gamma(TB)$, $\xi_2 \in \Gamma(TF)$. Therefore, by the consequences of lemma \ref{lem__3_2_} and proposition \ref{prop__3_1_}, we obtain $X(\ln f)\xi_{2}=0,\,\forall X\in \Gamma(TB)$ and $g(Z,\xi_2)grad(\ln f)=0,\,\forall Z \in \Gamma(TF)$, both of which implies that $\xi_{2}=0$ since $f$ is not constant. This completes the proof of the theorem.
\end{proof}
\noindent Let $M={}_{f_{2} } B \times_{f_{1} }F$ be a doubly warped product submanifold of a paracosymplectic manifold $\widetilde{M}(\phi,\xi,\eta,g)$ such that $\xi \in  \Gamma(TM)$. If we consider $\xi \in \Gamma(TF)$. Then from Eqs. $\eqref{GrindEQ__3_1_}$ and $\eqref{GrindEQ__3_2_}$, we obtain that $\nabla _{X} \xi= X(\ln f_{1})\xi + \xi(\ln f_{2})X =0$. This reduced to the equation $X(\ln f_{1})=0, \forall$ $X \in \Gamma(TB)$ which yields $f_1$ is constant. Again if we take $\xi \in \Gamma(TB)$ then $ Z(\ln f_{2})=0,\forall Z \in \Gamma(TF)$ this implies that $f_2$ is constant.
Therefore, we can state the following proposition:
\begin{proposition}$\label{prop__3_3_}$
Let $M={}_{f_{2} } B \times_{f_{1} } F $ be a doubly warped product submanifold of a paracosymplectic manifold $\widetilde{M}(\phi,\xi,\eta,g)$. Then $f_1$ (resp., $f_2$) is constant if $\xi \in \Gamma(TF) $ (resp., $\xi \in \Gamma(TB) $). 
 \end{proposition}
 \noindent As an immediate consequence of the proposition \ref{prop__3_3_}, we have 
\begin{corollary}$\label{coro__3_4_}$
Let $\widetilde{M}(\phi,\xi,\eta,g)$ be a paracosymplectic manifold. Then there do not exist a proper warped product submanifold $M=B\times_{f}F$ of $\widetilde{M}$ for $\xi \in \Gamma(TF)$.
\end{corollary}
\noindent Now we prove the following important lemma for later use:

\begin{lemma}$\label{lem__3_6_}$
Let $M=B \times_{f} F $ be a proper warped product submanifold of a paracosymplectic manifold $\widetilde{M}(\phi,\xi,\eta,g)$ such that $\xi \in \Gamma(TB)$. Then we have
\begin{align}
 &\xi(\ln f)=0,\label{GrindEQ__3_5_}\\
&A_{nZ}X=-t'h(X,Z),\label{GrindEQ__3_6_}\\
&g(h(X,W),nZ)=g(h(X,Z),nW)=-tX(\ln f)g(Z,W),\label{GrindEQ__3_8_}
\end{align}
for any $X,Y\in \Gamma(TB)$ and $Z,W\in \Gamma(TF)$.
\end{lemma}
\begin{proof}
 Equation \eqref{GrindEQ__3_5_} directly follows from Eq. \eqref{GrindEQ__3_2_} and proposition \ref{prop__3_1_}.
  Again by use of proposition \ref{prop__3_1_} and Eq.\eqref{GrindEQ__2_4_}, we obtain that  
 
 \begin{align}\label{1}\widetilde\nabla_X\phi Z-\phi(\widetilde\nabla_XZ)=0.\end{align}
On employing Eqs. \eqref{GrindEQ__2_5_}, \eqref{GrindEQ__2_8_} and \eqref{GrindEQ__2_9_} in Eq.\eqref{1}, we get 
\begin{equation}\label{GrindEQ__3_9_}
  h(X,tZ)-A_{nZ}X + \nabla^{\bot}_X{nZ}= t'h(X,Z)+n'h(X,Z).
 \end{equation}
By comparing the tangential part of $\eqref{GrindEQ__3_9_}$, we have Eq. \eqref{GrindEQ__3_6_}.
In view of Eqs. \eqref{GrindEQ__2_3_}, \eqref{GrindEQ__2_4_}, \eqref{GrindEQ__2_6_}, \eqref{GrindEQ__3_6_} and proposition \ref{prop__3_1_}, we achieve Eq. \eqref{GrindEQ__3_8_}.
This completes the proof of the lemma.
\end{proof}

\section{$\mathcal{P}\mathcal{R}$-semi-invariant warped product}\label{pr}
In \cite{CPR}, Chen-Munteanu defined $\mathcal{P}\mathcal{R}$-warped products in para-K\"{a}hler manifolds. Motivated to the study of Chen-Munteanu, we define $\mathcal{P}\mathcal{R}$-{\it semi-invariant warped product} submanifolds of an almost paracontact manifold. 
\begin{definition}
Let $M$ is an isometrically immersed pseudo-Riemannian submanifold of an almost paracontact manifold $\widetilde{M}(\phi,\xi,\eta,g)$. Then $M$ is said to be a  $\mathcal{P}\mathcal{R}$-{\it semi-invariant submanifold} if it is furnished with the pair of orthogonal distribution $(\mathfrak{D},\mathfrak{D}^\bot)$ satisfying the conditions:
\begin{itemize}
\item[(i)] $TM = \mathfrak{D}\oplus \mathfrak{D}{^\bot}\oplus <\xi> $,
\item[(ii)] the distribution $\mathfrak{D}$ is invariant under $\phi$, i.e., $\phi(\mathfrak{D})=\mathfrak{D}$ and
\item[(iii)] the distribution $\mathfrak{D}{^\bot}$ is anti-invariant under $\phi$, i.e., $\phi(\mathfrak{D}{^\bot})\subset TM^{\bot}$. 
\end{itemize}
A $\mathcal{P}\mathcal{R}$-semi-invariant submanifold is called a $\mathcal{P}\mathcal{R}$-{\it semi-invariant warped product} if it is a warped product of the form: $B\times_{f} F$ or $F\times_{f} B$, where $B$ is an invariant submanifold, $F$ is an anti-invariant submanifold of an almost paracontact manifold $\widetilde{M}(\phi,\xi,\eta,g)$ and  $f$ is a non-constant positive smooth function on the first factor. 
If the warping function $f$ is constant then a $\mathcal{P}\mathcal{R}$-semi-invariant warped product submanifold is said to be a $\mathcal{P}\mathcal{R}$-{\it semi-invariant product}.  
\end{definition}
\noindent In this section we shall examine $\mathcal{P}\mathcal{R}$-semi-invariant warped product submanifolds of a paracosymplectic manifold $\widetilde{M}$. 
\begin{proposition}\label{thrm__4_2_}
 There do not exist a $\mathcal{P}\mathcal{R}$-semi-invariant warped product submanifold $M=B \times_{f}F$ of a paracosymplectic manifold $\widetilde{M}(\phi,\xi,\eta,g)$ such that the characteristic vector field $\xi$ is tangent to $F$.
\end{proposition}
\begin{proof}
By the virtue of proposition \eqref{prop__3_1_} and Eq. \eqref{GrindEQ__3_2_}, we obtain that $\nabla_{X}\xi=\nabla_{\xi}X = X(\ln f)\xi=0,\,\forall X\in \Gamma(TB)$, which implies that $f$ is constant on $B$. This completes the proof.
\end{proof}
\begin{theorem}\label{thrm__4_1_}
 There do not exist a $\mathcal{P}\mathcal{R}$-semi-invariant warped product submanifold $M=B \times_{f}F$ of a paracosymplectic manifold $\widetilde{M}(\phi,\xi,\eta,g)$ such that the characteristic vector field $\xi$ is normal to $M$.
 \end{theorem}
\begin{proof}
Let $M=B \times_{f}F$ be a $\mathcal{P}\mathcal{R}$-semi-invariant warped product in $\widetilde{M}$ with $\xi\in \Gamma(TM^{\bot})$. Then for any $X\in \Gamma(TB)$ and $Z\in \Gamma(TF)$ we obtain from proposition \ref{prop__3_1_} that $\nabla_{X}Z=\nabla_{Z}X=X(\ln f)Z$, by taking the inner product with $Z$ and using Eqs. \eqref{GrindEQ__2_2_}, \eqref{GrindEQ__2_3_}, \eqref{GrindEQ__2_4_} and Gauss formula \eqref{GrindEQ__2_5_}, we get
\begin{align}
X(\ln f)\vert\vert Z\vert\vert^2=-g(h(Z,\phi X),\phi Z).\nonumber
\end{align}
Interchanging $Z$ by $\phi Z$, we have $X(\ln f)\vert\vert Z\vert\vert^2=0.$ This implies that $f$ is constant on $B$ since $Z$ is non-null vector field in $F$. This completes the proof of the theorem. 
\end{proof}
\begin{proposition}\label{thrm-4.7}
 There do not exist a $\mathcal{P}\mathcal{R}$-semi-invariant warped product submanifold $M=F \times_{f}B$ of a paracosymplectic manifold $\widetilde{M}(\phi,\xi,\eta,g)$ such that the characteristic vector field $\xi$ is tangent to $B$.
\end{proposition}
\begin{proof}
When $\xi\in \Gamma(TB)$, then by corollary \ref{coro__3_4_} we have simply a $\mathcal{P}\mathcal{R}$-semi-invariant warped product manifold.
This completes the proof.  
\end{proof}
\noindent Now, we give the following important result:
\begin{theorem}$\label{thrm__4_3_}$
Let $M = {F}\times_f {B}$ be a $\mathcal{P}\mathcal{R}$-semi-invariant warped product submanifold of a paracosymplectic manifold $\widetilde{M}(\phi,\xi,\eta,g)$ such that the characteristic vector field $\xi$ is tangent to $F$. Then the invariant distribution $\mathfrak{D}$ and the anti-invariant distribution $\mathfrak{D}^\bot$ are always integrable. 
\end{theorem}
\noindent Before going to the proof of this theorem, we first prove the following lemma:
\begin{lemma}
For a $\mathcal{P}\mathcal{R}$-semi-invariant warped product submanifold $M = {F}\times_f {B}$ of a paracosymplectic manifold $\widetilde{M}(\phi,\xi,\eta,g)$ with $\xi\in\Gamma(TF)$, we obtain for all  $U, V, Z \in \Gamma(\mathfrak{D})$ and $X , Y \in \Gamma(\mathfrak{D}^\bot)$ that
\begin{align}
&A_{nX}U = -X(\ln f)tU,\label{GrindEQ__4_4_}\\
&A_{nY}X=A_{nX}Y=t'h(X,Y)=0,\label{GrindEQ__4_9_}\\
&h(U,tV)=-g(U,V)n(grad(\ln f))+ n'h(U,V)=h(V,tU).\label{GrindEQ__4_11_}
\end{align}
\end{lemma}
\begin{proof}
It readily follows from Eqs.\eqref{GrindEQ__2_4_}, \eqref{GrindEQ__2_5_}, \eqref{GrindEQ__2_8_} and \eqref{GrindEQ__2_9_} that 
 $h(X,tU)= t'h(X, U)+n'h(X, U)$. This equation yields by comparing the tangential parts that $t'h(X, U) =0$ and by making use of equation \eqref{GrindEQ__2_16_}, we have the formula \eqref{GrindEQ__4_4_}.  Since, the distribution $\mathfrak{D}^\bot$ is totally geodesic in $M$ and anti-invariant then from Eqs. \eqref{GrindEQ__2_5_} and \eqref{GrindEQ__2_6_}, we get 
 \begin{align}\label{GrindEQ__4_5_}
 -A_{nY}X+\nabla^{\bot}_X {nY}=n(\nabla_X{Y})+t'h(X, Y)+n'h(X, Y).
 \end{align}
By equating the tangential components  of Eq. \eqref{GrindEQ__4_5_} and then interchanging $X$ to $Y$, we obtain that
\begin{align}\label{GrindEQ__4_7_}
A_{nY}X=A_{nX}Y=-t'h(X,Y).
\end{align} 
Employing Eqs. \eqref{GrindEQ__2_3_}, \eqref{GrindEQ__2_4_}-\eqref{GrindEQ__2_7_}, proposition \ref{prop__3_1_} and using the fact that $A$ is self-adjoint, we attain 
 \begin{align}\label{GrindEQ__4_8_}
 g(A_{nX}Y,Z)=-g(A_{nY}X,Z).
 \end{align}
 In light of Eqs.\eqref{GrindEQ__4_7_} and \eqref{GrindEQ__4_8_}, we obtain formula \eqref{GrindEQ__4_9_}.
 On the other side we obtain 
\begin{align}\label{other}h(U,tV)+ {\nabla }_U tV=t(\nabla^{'}_{U} V)-g(U,V)n(grad(\ln f))+ t'h(U,V)+n'h(U,V).\end{align}
By equating the normal components of Eq. \eqref{other}, we get formula \eqref{GrindEQ__4_11_}. This completes the proof.
\end{proof}

\noindent{\it Proof of Theorem $\ref{thrm__4_3_}$.}
Let $U, V\in \Gamma(\mathfrak{D})$ then by the virtue of Eqs.\eqref{GrindEQ__2_17_}, \eqref{GrindEQ__4_11_} and using the fact that $h$ is symmetric, we have
 \begin{align*}
 n([V,U])&=n(\nabla_VU)-n(\nabla_UV)=\nabla^{\bot}_V{nU}-(\nabla_Vn)U-\nabla^{\bot}_U{nV}+(\nabla_Un)V\\
         &=(\nabla_Un)V-(\nabla_Vn)U\\ &=n'h(U,V)-h(U,tV)-n'h(U,V)+h(V,tU)=0,
 \end{align*}
 this implies that $[V,U] \in \Gamma(\mathfrak{D})$ for any $U$, $V \in \mathfrak{D}$.  Similarly, by using Eqs. \eqref{GrindEQ__2_16_} and \eqref{GrindEQ__4_7_} 
 we find that $t([X,Y])=0$ which implies $[X,Y] \in \Gamma(\mathfrak{D}^\bot)$ for any $X$, $Y \in  \Gamma(\mathfrak{D}^\bot)$. Thus the distributions $\mathfrak{D}$ and $\mathfrak{D}^\bot$ are integrable. This complete the proof of the theorem. \qed
\begin{theorem}$\label{thrm__4_4_}$
Let $M\to \widetilde{M} $ be an isometric immersion of a pseudo-Riemannian manifold $M$ into a paracosymplectic manifold $\widetilde{M}(\phi,\xi,\eta,g)$ such that the characteristic vector field $\xi$ is tangent to $M$. Then a necessary and sufficient condition for $M$ to be a $\mathcal{P}\mathcal{R}$-semi-invariant submanifold is that $n\circ t=0$.
\end{theorem}
\begin{proof}
 If we denote the orthogonal projections on the invariant distribution $\mathfrak{D}$ and the anti-invariant distribution $\mathfrak{D}^\bot$ by $P_1$ and $P_2$ respectively. Then we have 
 \begin{align}\label{GrindEQ__4_13_}
  P_1+P_2=Id, \, (P_1)^2=P_1, \,(P_2)^2=P_2,\, P_1P_2=P_2P_1=0.
 \end{align}
Since $\xi \in \Gamma(TM)$, then for any $X \in \Gamma(TM)$, $Z \in \Gamma(TM^{\bot})$ we obtain
 \begin{align}
  X-\eta(X)\xi&=t^2X +t'nX, \label{A}\\  ntX + n'nX&=0,\label{B}\\ tt'Z+t'n'Z&=0,\label{C}\\ nt'Z+n'^2Z&=Z.\label{D}
 \end{align}
 From Eqs. \eqref{GrindEQ__2_1_} and \eqref{GrindEQ__2_8_}, we can write
 \begin{align*}
 &X-\eta(X)\xi =P_1X+P_2X,\\
 &\phi X =\phi(P_1X)+\phi(P_2X),\\
 &tX+nX =tP_1X+nP_1X+tP_2X+nP_2X
 \end{align*}
 for any $X \in \Gamma(TM)$. By comparing the tangential and the normal parts of last equation, we find
 \begin{align}\label{GrindEQ__4_15_}
 tX =tP_1X+tP_2X, \quad nX =nP_1X+nP_2X.
 \end{align}
For the  invariant distribution $\mathfrak{D}$ and the anti-invariant distribution $\mathfrak{D}^\bot$,
 we obtain that $ nP_1=0$ and $tP_2=0$. Thus from Eq. \eqref{GrindEQ__4_15_}, we have
 \begin{align*}t=tP_1,\, n =nP_2\end{align*}
 which gives
  \begin{align*} ntX=nP_2 tX=nP_2tP_1X=nt(P_1P_2)X=0,\,\forall X \in \Gamma(TM). \end{align*}
 
\noindent Conversely, suppose that $M$ be submanifold of a paracosymplectic manifold $\widetilde{M}$ such that $\xi\in\Gamma(TM)$ satisfying $nt=0$.
Then from Eq. \eqref{B}, we have
\begin{align}\label{GrindEQ__4_16_}
 n'n=0.\end{align} 
Employing Eqs. \eqref{GrindEQ__2_3_}, \eqref{B} and \eqref{GrindEQ__4_16_}, we obtain that $g(X,tt'Z)=0$ for any $X \in \Gamma(TM)$ and $Z \in \Gamma(TM^{\bot})$ which implies that $tt'=0.$ Therefore from Eq. \eqref{C}, we also have $t'n'=0$. Further, from Eqs. \eqref{A} and \eqref{D}, we get 
\begin{align}
t^3 = t, \quad n'^3 = n'. \label{GrindEQ__4_17_}
\end{align}
\noindent By substituting
\begin{align}
P_1=t^2 \,{\rm{and}}\,P_2 = Id-t^2, \label{GrindEQ__4_18_}
\end{align}
\noindent we achieve Eq. \eqref{GrindEQ__4_13_}, this implies $P_1$ and $P_2$ are orthogonal complementary projections defining distributions $\mathfrak{D}$ and $\mathfrak{D}^\bot$. From Eqs. \eqref{GrindEQ__4_17_} and \eqref{GrindEQ__4_18_}, we determine that $t =tP_1,\, tP_2=0,\, n = nP_2,\, nP_1=0\,\,{\rm and}\,\, P_2tP_1=0$. These implies that the distribution $\mathfrak{D}$ is invariant and the distribution $\mathfrak{D}^\bot$ is anti-invariant, and hence completes the proof of the theorem.
\end{proof}
\begin{theorem}$\label{thrm__4_5_}$
Let $M$ be a $\mathcal{P}\mathcal{R}$-semi-invariant submanifold of a paracosymplectic manifold $\widetilde{M}(\phi,\xi,\eta,g)$. Then $M$ is a $\mathcal{P}\mathcal{R}$-semi-invariant warped product ${F}\times {}_f{B}$ iff the shape operator of $M$ satisfies
\begin{align}\label{GrindEQ__4_19_}
A_{\phi X}U=-X(\mu)\phi U,\quad X \in \Gamma(\mathfrak{D}^\bot), U \in \Gamma(\mathfrak{D}) 
\end{align}
\noindent for some function $\mu$ on $M$ such that $W(\mu)=0$, $W \in \Gamma(\mathfrak{D})$.
\end{theorem}
\begin{proof}
 Let $M = {F}\times {}_f{B}$ be a $\mathcal{P}\mathcal{R}$-semi-invariant warped product submanifold of a paracosymplectic manifold $\widetilde{M}$ then from Eq.\eqref{GrindEQ__4_4_}, we obtain that $A_{\phi X}U=-X(\ln f)\phi U$ for any $X \in \Gamma(\mathfrak{D}^\bot)$ and $U \in \Gamma(\mathfrak{D})$. Since $f$ is a function on the first factor $F$, putting $\mu=\ln f$ implies that $W(\mu)=0$ for all $W \in \Gamma(\mathfrak{D}).$ Conversely, assume that $M$ satisfies \eqref{GrindEQ__4_19_} for some function $\mu$ on $M$ with $W(\mu)=0$, for all $W \in \Gamma(\mathfrak{D})$. By the virtue of Eqs. \eqref{GrindEQ__2_3_}, \eqref{GrindEQ__2_4_}-\eqref{GrindEQ__2_6_} and \eqref{GrindEQ__4_9_}, we have
 \begin{align*}
 g(\nabla_XY, \phi V)=g(\widetilde\nabla_XY,\phi V)=-g(\widetilde\nabla_X\phi Y,V)=-g(A_{\phi Y}X,V)=0,
 \end{align*}
 for any $X,Y \in \Gamma(\mathfrak{D}^\bot)$ and $V \in \Gamma(\mathfrak{D})$. Therefore the distribution $\mathfrak{D}^\bot$ is totally geodesic.  On the other hand from Eqs.\eqref{GrindEQ__2_2_} and \eqref{GrindEQ__4_4_}, we get 
\begin{align*}
 g(\nabla_UV, X)&=g(\widetilde\nabla_UV,X)=-g(V,\widetilde\nabla_UX)=g(\phi V,\widetilde\nabla_U{\phi X})\\
 &=-g(\phi V,A_{nX}U)=X(\mu)g(\phi V,\phi U)=-X(\mu)g(U,V),
 \end{align*}
\noindent for any $U,V \in \Gamma(\mathfrak{D})$, where $\mu = \ln f$. Thus, the integrable manifold of $\mathfrak{D}$ is totally umbilical submanifold in $M$ and its mean curvature is non-zero and parallel by using the  facts that the distribution $\mathfrak{D}$ of $M$ is always integrable and $W(\mu)=0$ for all $W \in \Gamma(TB)$, and hence  completes the proof of the theorem.
\end{proof}

\noindent Now, we prove the following result:
\begin{theorem}$\label{thrm__4_6_}$
Let $M\to \widetilde{M} $ be an isometric immersion of a pseudo-Riemannian manifold $M$ into a paracosymplectic manifold $\widetilde{M}(\phi,\xi,\eta,g)$. Then a necessary and sufficient condition for $M$ to be a $\mathcal{P}\mathcal{R}$-semi-invariant warped product ${B}\times_f{F}$ submanifold is that the shape operator of $M$ satisfies
\begin{align}\label{eq4.6}
A_{\phi Z}X=-\phi X (\mu)Z,\,X \in \Gamma(\mathfrak{D}\oplus<\xi>), Z \in \Gamma(\mathfrak{D}^{\bot}), 
\end{align}
for some function $\mu$ on $M$ such that $V(\mu)=0$, $V \in \Gamma(\mathfrak{D}^{\bot})$.
\end{theorem}

\begin{proof}
 Let $M = {B}\times_f{F}$ be a $\mathcal{P}\mathcal{R}$-semi-invariant warped product submanifold of a paracosymplectic manifold $\widetilde{M}$ such that $\xi\in\Gamma(TB)$. Then from Eq. \eqref{GrindEQ__3_8_}, we accomplish that $g(A_{\phi Z}W,X)=-(\phi X \ln f)g(W,Z)$ which implies Eq. \eqref{eq4.6}. Since $f$ is a function on $B$, we also have $V(\ln f)=0$ for all $V \in \Gamma(\mathfrak{D}^{\bot})$. Conversely, suppose that $M$ satisfies Eq. \eqref{eq4.6} for some function $\mu$ with $V(\mu)=0$ for all $V \in \Gamma(\mathfrak{D}^{\bot})$. Then we have
 \begin{align}\label{gh}
 g(h(X,Y),\phi Z)=0,
 \end{align}
 by use of Eqs. \eqref{GrindEQ__2_3_}, \eqref{GrindEQ__2_5_} and the fact that $\widetilde{M}$ is a paracosymplectic manifold, we attain that
 \begin{align}\label{gh1}
 g(\widetilde\nabla_{X}\phi Y,\phi Z)=g(\phi\widetilde\nabla_{X}Y, \phi Z)=-g(\widetilde\nabla_{X}Y, Z)=-g(\nabla_{X}Y, Z)=0,\end{align}
 for any $X,\,Y \in \Gamma(\mathfrak{D}\oplus<\xi>), Z \in \Gamma(\mathfrak{D}^{\bot})$. This means that the distribution $(\mathfrak{D}\oplus<\xi>)$ is integrable and its leaves are totally geodesic in $M$. On the other hand, let $F$ be a leaf of $\mathfrak{D}^{\bot}$ and $\mathfrak{h}$ be the second fundamental form of the immersion of $F$ into
$M$ then for any $Z ,W\in \Gamma(\mathfrak{D}^{\bot})$, we obtain by using Eqs. \eqref{GrindEQ__2_4_}, \eqref{GrindEQ__2_7_} and \eqref{eq4.6} that
\begin{align}\label{eq4.6.0}
\mathfrak{h}(Z,W)=g(Z, W)\nabla\mu,
\end{align}
where $\nabla\mu$ is the gradient of the function $\mu$. Then it follows from \eqref{eq4.6.0} that the
leaves of $\mathcal{D}^{\bot}$ are totally umbilical in $M$. Also, for any $V \in \Gamma(\mathfrak{D}^{\bot})$ , we have $V(\mu)=0$, which implies that the integral manifold of $\mathfrak{D}^{\bot}$ is an extrinsic sphere in $M$, i.e., a totally umbilical submanifold with parallel mean curvature vector. Thus, by \cite{SH} we achieve that $M$ is a $\mathcal{P}\mathcal{R}$-semi-invariant submanifold of a paracosymplectic manifold $\widetilde{M}$. This completes the proof of the theorem.
\end{proof}
\section{Example}\label{ex}
In this section, we present an example for a $\mathcal{P}\mathcal{R}$-semi-invariant submanifold of a paracosymplectic manifold in the form $F\times_{f} B$:

\begin{example}
 Let $\widetilde M=\mathbb{R}^4\times\mathbb{R}_{+}\subset\mathbb{R}^5$ be a $5$-dimensional manifold with the standard Cartesian coordinates $(x_1,x_2,y_1,y_2,t)$. Define the paracosymplectic pseudo-Riemannian metric structure $(\phi,\xi,\eta, g)$ on $\widetilde M$ by 
 \begin{align}
 &\phi e_{1}=e_{3},\ \phi e_{2}=e_{4},\ \phi e_{3}=e_{1},\  
 \phi e_{4}=e_{2},\ \phi e_{5}=0, \label{strcex2.1}\\
  &\xi=e_{5},\ \eta =dt,\  g=\sum_{i=1}^{2}(dx_{i})^{2}-\sum_{j=1}^{2}(dy_{j})^{2}+\eta\otimes\eta.\label{strcex2.2} 
 \end{align}
Here, $\left\{e_1,e_2,e_3,e_4,e_5\right\}$ is a local orthonormal frame for $\Gamma(T\widetilde M)$.
Let $M$ is an isometrically immersed pseudo-Riemannian submanifold of a paracosymplectic manifold $\widetilde M$ given by
\begin{align*}
 \Omega(v,\theta,\beta,u)=(v\tan\theta,v\tan\beta,v\sec\theta,v\sec\beta,u),
\end{align*}
 where $\theta \in (0,\pi/2)$, $\beta \in (0,\pi/2 )$ and $v$ is non-zero.
 Then the tangent bundle of $M$ is spanned by the vectors
 \begin{align}\label{tanbundl} 
 X_{1}&=\tan(\theta) e_{1}+\tan(\beta) e_{2}+\sec(\theta) e_{3}+\sec(\beta) e_{4}, \nonumber\\
 X_{2}&=v\sec^{2}(\theta) e_{1}+v\sec(\theta)\tan(\theta) e_{3}, \\
 X_{3}&=v\sec^{2}(\beta)e_{2}+v\sec(\beta)\tan(\beta)e_{4},\, X_{4}=e_{5}.\nonumber 
   \end{align}
 The space $\phi(TM)$ with respect to the paracosymplectic pseudo-Riemannian metric structure $(\phi,\xi,\eta, g)$ of $\widetilde M$ becomes
  \begin{align}\label{norbndl} 
\phi(X_{1})&=\sec(\theta)e_{1}+\sec(\beta)e_{2}+\tan(\theta)e_{3}+\tan(\beta)e_{4}, \nonumber\\
 \phi(X_{2})&=v\sec(\theta)\tan(\theta)e_{1}+v\sec^{2}(\theta)e_{3}, \\
 \phi(X_{3})&=v\sec(\beta)\tan(\beta)e_{2}+v\sec^{2}(\beta)e_{4},\,\phi(X_{4})=0. \nonumber
   \end{align}
 From Eqs.\eqref{tanbundl} and \eqref{norbndl} we obtain that $\phi(X_{4})$ is orthogonal to $M$, and $\phi(X_{1})$, $\phi(X_{2})$, $\phi(X_{3})$ are tangent to $M$. So $\mathfrak{D}^{\bot}$ and $\mathfrak{D}$ can be taken as a subspace span\{$X_{4}$\} and a subspace span\{$X_{1}, X_{2}, X_{3}$\} respectively, where $\xi = X_{4}$ for $\phi(X_{4})=0$ and $\eta({X_{4}})=1$. Therefore, $M$ becomes a $\mathcal{P}\mathcal{R}$-semi-invariant submanifold. Further, the induced pseudo-Riemannian metric tensor $g$ of $M$ is given by
 \begin{align*}
[g(e_{i},e_{j})]=\begin{bmatrix}
                         -2&0&0&0\\
                         0&v^{2}\sec^{2}\theta&0&0\\
                         0&0&v^{2}\sec^{2}\beta&0\\
                         0&0&0& 1
\end{bmatrix},
\end{align*}
that is, 
\begin{align*} g=du^{2}+v^{2}\{\sec^{2}(\theta) (d\theta)^{2}+ \sec^{2}(\beta) (d\beta)^{2}-(2/v^{2})dv^{2}\}=g_{F}+v^{2}g_{B}.\end{align*}
 \end{example}
 \noindent Hence, $M$ is a $4$-dimensional $\mathcal{P}\mathcal{R}$-semi-invariant warped product submanifold of $\widetilde{M}$ with warping function $f=v^{2}$.

\end{document}